\documentclass[12pt]{amsart}
\usepackage[top=.85 in,bottom=.85 in, left=.85 in, right= .85 in]{geometry}
\numberwithin{equation}{section}

\usepackage{dsfont,amsfonts,amsmath}

\newcommand{\eps}{\varepsilon}

\newcommand{\RR}{\mathds{R}}

\usepackage{fancybox}
      \newtheorem{theorem}{Theorem}[section]
       \newtheorem{proposition}[theorem]{Proposition}
       \newtheorem{corollary}[theorem]{Corollary}

       \newtheorem{remark}{Remark}[section]

\def\E{{\mathbb E}}
\def\V{{\rm Var}}
\def\Var{{\rm Var}}

\date{Created:  October 27, 2009. \\ Printed: \today \ file: \jobname.tex}

\author{
W{\l}odek  Bryc
}

\address{
Department of Mathematical Sciences,
University of Cincinnati,
PO Box 210025,
Cincinnati, OH 45221--0025, USA}
\title[Quadratic Harnesses]{Quadratic harnesses from generalized beta integrals
}

\begin{document}

 \begin{abstract}
We use generalized beta integrals to construct examples of Markov processes with linear regressions, and quadratic second  conditional moments.
  \end{abstract}
 \maketitle

\section{Introduction}

\subsection{Quadratic harnesses}\label{S: QH}
In \cite{Bryc-Matysiak-Wesolowski-04} the authors  consider
 square-integrable  stochastic processes on $(0,\infty)$
such that
for all $t,s> 0$,
\begin{equation}\label{EQ: cov}
\E(X_t)=0,\: \E(X_tX_s)=\min\{t,s\},
\end{equation}
$\E({X_t}|{\mathcal{F}_{ s, u}})$ is a linear function of $X_s,X_u$, and
$\V [X_t|\mathcal{F}_{s,u }]$ is a quadratic  function of $X_s,X_u$.
Here, $\mathcal{F}_{ s, u}$ is the two-sided
$\sigma$-field generated by $\{X_r: r\in (0,s]\cup[u,\infty)\}$.
Then  for all $s<t<u$, \eqref{EQ: cov} implies that
\begin{equation}
\label{EQ: LR} \E({X_t}|{\mathcal{F}_{ s, u}})=\frac{u-t}{u-s}
X_s+\frac{t-s}{u-s} X_u,
\end{equation}
which is sometimes referred to as a harness condition, see \cite{Mansuy-Yor04}. While there are numerous examples of harnesses that include all integrable L\'evy processes (\cite[(2.8)]{Jacod-Protter-88}),
the  assumption of quadratic conditional variance is more restrictive, see
\cite{Wesolowski93}.
Under certain  assumptions,  \cite[Theorem 2.2]{Bryc-Matysiak-Wesolowski-04}
 asserts that there exist numerical constants
 $\eta,\theta\in\RR$ $\sigma,\tau>0$ and $ \gamma\in[-1, 1+2\sqrt{\sigma\tau}]$ such that
for all $s<t<u$,
\begin{multline}\label{EQ: q-Var}
\V [X_t|\mathcal{F}_{s,u }]
= \frac{(u-t)(t-s)}{u(1+\sigma s)+\tau-\gamma s}\left( 1+\eta \frac{uX_s-sX_u}{u-s} +\theta\frac{X_u-X_s}{u-s}\right. \\
 \left.
+ \sigma
\frac{(uX_s-sX_u)^2}{(u-s)^2}+\tau\frac{(X_u-X_s)^2}{(u-s)^2}
-(1-\gamma)\frac{(X_u-X_s)(uX_s-sX_u)}{(u-s)^2} \right).
\end{multline}
We will say that a square-integrable stochastic process $(X_t)_{t\in T}$ is a quadratic harness on $T$ with parameters $(\eta,\theta,\sigma,\tau,\gamma)$,
if it satisfies \eqref{EQ: cov}, \eqref{EQ: LR} and \eqref{EQ: q-Var} on an open interval $T\subset (0,\infty)$.

Our goal is to construct examples of Markov quadratic harnesses with $\gamma=1-2\sqrt{\sigma\tau}$. In \cite[Proposition 4.4]{Bryc-Matysiak-Wesolowski-04}, these were called "classical quadratic harnesses. The construction follows   \cite[Section 2]{Bryc-Wesolowski-08} who construct quadratic harnesses with $\gamma<1-2\sqrt{\sigma\tau}$ from the Askey-Wilson integral.
Here we use instead some of the generalized Beta integrals %and generalized Beta sums
from \cite{Askey:1989}.

The paper is organized into sections  based on the number of parameters in the generalized beta integrals. In particular, in  Section \ref{Sec:2par} we   exhibit explicit transition probabilities for the bridges of the hyperbolic secant  process, and for completeness in Section \ref{Sect:SBI} we re-analyze the Dirichlet process.

\subsection{Conversion to the standard form}
In this section we recall a procedure that we use to transform Markov processes with linear regressions and quadratic conditional variances into the  quadratic harnesses.
The following is a specification  of \cite[Theorem 1.1]{Bryc-Wesolowski-09} that fits our needs.
\begin{proposition}\label{P-Mobius}
Suppose $(Y_t)$ is a (real-valued) Markov process on an open interval $T\subset\RR$ such that
\begin{enumerate}
\item $\E(Y_t)=\alpha+\beta t$ for some real $\alpha,\beta$.
\item For $s<t$ in $T$, ${\rm Cov}(Y_s,Y_t)=M^2(\psi+s)(\delta+\eps t)$, where    $M^2(\psi+t)(\delta+\eps t)>0$ on the entire interval $T$, and that $\delta-\eps\psi>0$.
\item For $s<t<u$,
\begin{equation}\label{computed-cond-var}
\V(Y_t|Y_s,Y_u)=F_{t,s,u}
\left(\chi_0+\eta_0\frac{uY_s-sY_u}{u-s}+\theta_0 \frac{Y_u-Y_s}{u-s}+\frac{(Y_u-Y_s)^2}{(u-s)^2}\right),
\end{equation}
where $F_{t,s,u}$ is non-random and $\chi_0,\theta_0,\eta_0\in\RR$  are such that $\chi:=\chi_0+\alpha
\eta_0+\beta\theta_0+\beta^2> 0$.
\end{enumerate}
Denote $\widetilde Y_t=Y_t-\E(Y_t)$. Then there are two affine functions
$\ell(t)=\frac{t \delta -\psi }{M (\delta - \epsilon
   \psi) }$ and $m(t)=
 \frac{1-t \epsilon }{M (\delta -\epsilon  \psi)
   }$ and an open interval $T'\subset(0,\infty)$ such that
$X_t:=m(t)\widetilde Y_{\ell(t)/m(t)}$ defines a process $(X_t)$ on $T'$  such that \eqref{EQ: cov} holds and  \eqref{EQ: q-Var} holds with parameters
\begin{eqnarray}
\eta&=& M \left(\delta  \eta _0+\epsilon  \left(2
   \beta +\theta _0\right)\right)/\chi \,,  \\
\theta&=&   M \left(2 \beta +\psi  \eta _0+\theta
   _0\right)/\chi\,, \\
\sigma&=& M^2\eps^2/\chi\,, \\
\tau&=&\ M^2/\chi\,,\\
\gamma&=& 1+2\eps \sqrt{\sigma\tau}\,.
\end{eqnarray}
\end{proposition}
\begin{proof} This is \cite[Theorem 1.1]{Bryc-Wesolowski-09} specialized to
$\chi = \chi _0$,  $\eta =
   \eta _0$, $\theta = \theta _0$, $\sigma
   = 0$, $\tau =1$, $\rho = 0$, $a= M$, $b=M \psi $, $c= M
   \epsilon $, $d= M \delta $.
\end{proof}
\begin{remark}
We will apply this only to $\eps=0,\pm 1$, and $\chi_0,\theta_0,\eta_0\in\{0,1\}$.
\end{remark}
\begin{remark}
For $\eps\leq 0$, we see that $\gamma\leq 1$ and
 $\eta\sqrt{\tau}+\theta \sqrt{\sigma}= M^2(\delta-\eps\psi)\eta_0/\chi^2$ has the same sign as $\eta_0$.
\end{remark}
\begin{remark} The time domain $T'$ is the image of $T$ under the M\"obius transformation $t\mapsto (t+\psi)/(\eps t+\delta)$.
\end{remark}

Two related transformations are sometimes useful to keep in mind, as they take care of some additional non-uniqueness in the final form of \eqref{EQ: q-Var}.
Firstly, if $(X_t)$ is a quadratic harness with parameters $(\eta,\theta,\sigma,\tau,\gamma)$ then $(aX_{t/a^2})$ is a quadratic harness with parameters $ (\eta/a,a\theta,\sigma/a^2,a^2\tau,\gamma)$. In particular, if $\sigma=0$ and $\tau>0$,  then without loss of generality we may take $\tau=1$. And if $\sigma,\tau>0$ then without loss of generality we may take $\sigma=\tau$. (So our constructions will lead to these two cases only.)

Secondly, time inversion $(tX_{1/t})$ converts  a quadratic harness with parameters $(\eta,\theta,\sigma,\tau,\gamma)$ into a quadratic harness with parameters $(\theta,\eta,\tau,\sigma,\gamma)$, i.e. it swaps the entries within pairs $(\eta,\theta)$ and $(\sigma,\tau)$. In particular, time inversion maps a quadratic harness with $\sigma=0$, $\tau=1$ into a quadratic harness with $\sigma=1$, $\tau=0$. Similarly, it maps a quadratic harness with parameters  $\sigma=\tau$ and $\eta^2<4\sigma$, $\theta^2\geq 4\sigma$ into a quadratic harness with parameters $\sigma=\tau$ and $\eta^2\geq 4\sigma$, $\theta<4\sigma$.

\section{Four-parameter beta integral}\label{Sec:4par}
This section contains the construction of Markov processes based on the four-parameter beta integral \cite[(8.i)]{Askey:1989}. After a transformation, these processes become quadratic harnesses with arbitrary $\sigma=\tau\in(0,1)$, $\gamma=1-2\sqrt{\sigma\tau}$, and with $\eta,\theta$ such that  $\sqrt{\tau}\eta+\sqrt{\sigma}\theta\ne 0$; parameters  $\eta,\theta$ will be required to satisfy also some additional restrictions, of which $\eta\theta\geq 0$ suffices for all constructions to go through.
 Since the main steps will be repeated several times, first with three parameters to cover the case $\sigma=0$, and then with two parameters to cover the case $\sqrt{\tau}\eta+\sqrt{\sigma}\theta=0$, we give here more details so that we can suppress them in the subsequent iterations.

The construction starts with four complex numbers  $a_1,a_2,a_3,a_4$  with strictly positive real parts.  The generalized beta integral    \cite{deBranges:1972,Wilson:1980} after changing the variable to $\sqrt{x}$ is
\begin{equation}\label{WI}
\int_0^\infty \frac{\prod_{j=1}^4 \left(\Gamma(a_j+i\sqrt{x})\Gamma(a_j-i\sqrt{x})\right)
}{\sqrt{x}\left|\Gamma(2i\sqrt{x})\right|^2} dx
=\frac{4 \pi  \prod_{1\leq k<j\leq 4}\Gamma (a_k+a_j)}{\Gamma (a_1+a_2+a_3+a_4)}.
\end{equation}
%\begin{multline}\label{WI}
%\int_0^\infty
%%\frac{\Gamma(a+i\sqrt{x})\Gamma(a-i\sqrt{x})
%%\Gamma(b+i\sqrt{x})\Gamma(b-i\sqrt{x}) \Gamma(c+i\sqrt{x})\Gamma(c-i\sqrt{x}) \Gamma(d+i\sqrt{x})\Gamma(d-i\sqrt{x})
%\frac{\prod_{j=1}^4 \left(\Gamma(a_j+i\sqrt{x})\Gamma(a_j-i\sqrt{x})\right)
%}{\sqrt{x}\left|\Gamma(2i\sqrt{x})\right|^2} dx
%=\frac{4 \pi  \prod_{1\leq k<j\leq 4}\Gamma (a_k+a_j)}{\Gamma (a_1+a_2+a_3+a_4)}.
%\\=\frac{4 \pi  \Gamma (a+b) \Gamma (a+c) \Gamma (b+c) \Gamma (a+d) \Gamma (b+d) \Gamma (c+d)}{\Gamma (a+b+c+d)}.
%\end{multline}
Denote
\begin{equation}\label{K}
K(a,b,c,d)=\frac{\Gamma (a+b+c+d)}{4 \pi  \Gamma (a+b) \Gamma (a+c) \Gamma (b+c) \Gamma (a+d) \Gamma (b+d) \Gamma (c+d)}.
\end{equation}
If $a,b,c,d$ are   positive real numbers, or come as one or two conjugate pairs with positive real parts,  identity  \eqref{WI} implies that the following function of $x>0$ becomes a four-parameter probability density function on $(0,\infty)$:
\begin{equation}\label{f}
f(x;a,b,c,d)=K(a,b,c,d)\frac{\left|\Gamma(a+i\sqrt{x})
\Gamma(b+i\sqrt{x})\Gamma(c+i\sqrt{x}) \Gamma(d+i\sqrt{x})\right|^2 }{\sqrt{x}\left|\Gamma(2 i\sqrt{x})\right|^2}.
\end{equation}
%This density corresponds to the Askey-Wilson density (weight) function that appears in \cite[(2.4)]{Bryc-Wesolowski-08},
%and has analogous properties.

\begin{proposition}\label{P-abcd}
If a random variable $X$ has density $f(x;a,b,c,d)$ then
\begin{equation}\label{m-abcd}
\E(X)=\frac{abc +abd+acd+bcd}{a+b+c+d}
\end{equation}
and
\begin{equation}\label{Var-abcd}
\Var(X)=\frac{(a+b) (a+c) (b+c) (a+d) (b+d) (c+d)}{(a+b+c+d)^2 (a+b+c+d+1)}.
\end{equation}
\end{proposition}
\begin{proof}  The formulas can be read out from the first two orthogonal polynomials \cite[(1.1.4)]{Koekoek-Swarttouw}, but they also follow easily from the  formulas
$$\E(a^2+X)=\E\left((a+i\sqrt{X})(a-i\sqrt{X})\right)=\frac{K(a,b,c,d)}{K(a+1,b,c,d)}$$
and
$$
a^2b^2+(a^2+b^2)\E(X)+\E(X^2)=\E\left((a^2+X)(b^2+X)\right)=\frac{K(a,b,c,d)}{K(a+1,b+1,c,d)}.
$$
Now using \eqref{K} and $s\Gamma(s)=\Gamma(s+1)$, we get \eqref{m-abcd}, and  after a  calculation we get \eqref{Var-abcd}.
\end{proof}
Next, we prove a "convolution formula" which will be used to verify the Chapman-Kolmogorov equations.
\begin{proposition}\label{P-CK}
If $m>0$ then
\begin{multline}\label{CK1}
f(y;a,b,c+m,d+m)=\int_0^\infty f\left(y;a,b, m+i\sqrt{x},m-i\sqrt{x}\right)f(x;a+m,b+m,c,d)\,dx.
 \end{multline}
\end{proposition}
\begin{proof}
Re-arranging the factors in \eqref{f}, %with $\mu$ denoting the right hand side of \eqref{m-abcd}
we have
\begin{multline}\label{f/q}
\frac{f(x,a+m,b+m,c,d) f\left(y,a,b,m+i \sqrt{x},m-i \sqrt{x}\right)}{f(y,a,b,c+m,d+m)}=f\left(x,m+i
   \sqrt{y},m-i \sqrt{y},c,d\right).
\end{multline}
Formula \eqref{CK1} now follows, as $\int_0^\infty f\left(x; \mu+i\sqrt{y},\mu-i\sqrt{y},c,d\right)dx=1$.
\end{proof}
We remark that \eqref{f/q} is an analog of \cite[(b3)]{Jamison:1974} and will serve similar purposes. Related formulas  will appear again as \eqref{f/q3}, \eqref{f/q2}, \eqref{BBBridge}, and \eqref{f/q5}.

\subsection{The  auxiliary Markov process}\label{Sec:MarkovABCD}
We now define a family of Markov processes $(Y_t)_{t\in T}$, parameterized by $A,B,C,D$ that are either all real and positive or come as one or two complex conjugate pairs $A=\bar B$ or $C=\bar D$, with positive real parts.
Without loss of generality we may assume that $\Re( A) \leq \Re (B) $ and $\Re( C)\leq \Re (D)$.

%
%\begin{lemma} Under the above assumptions,
%% on $A,B,C,D$ made at the beginning of this section,
%\begin{equation}\label{SS}
%S^2=\frac{(A+C) (B+C) (A+D) (B+D) }{(A+B+C+D)^2 (A+B+C+D+1)}>0.
%\end{equation}
%\end{lemma}
% \begin{proof}
% Clearly, the denominator is a positive real number. We now considering the numerator.
% If $A,B,C,D>0$, the numerator is positive. If there is one conjugate pair, say $A=\alpha+i\beta$, $B=\alpha-i\beta$, $C,D>0$ then the numerator is
%$$\left((\alpha+C)^2+\beta^2\right)\left((\alpha+D)^2+\beta^2\right)>0.$$
%If in addition $C=\gamma+i\delta$, $D=\bar C$ then the numerator is
%$$\left((\alpha+\gamma+i\delta)^2+\beta^2\right)\left((\alpha+\gamma-i\delta)^2+\beta^2\right) $$
%$$=
%2 \gamma  (2 \alpha +\gamma ) \delta ^2+\left(\alpha ^2-\beta ^2+\delta ^2\right)^2+\left(2 \alpha ^2+2 \gamma  \alpha +\gamma ^2\right) \left(2 \beta ^2+\gamma
%   (2 \alpha +\gamma )\right)>0
%$$
%as $\alpha,\gamma>0$ by assumption.
%\end{proof}

As the time   domain for Markov process $(Y_t)$ we take the open interval $T=(-\Re (C), \Re( A))$,
 and as the state space we take $(0,\infty)$.
 We define the univariate distribution of $Y_t$  by the density
\begin{equation}\label{f-M}
f_t(x)=f(x;A-t, B-t,C+t,D+t).
\end{equation}
For $s<t$, we define the transition probability $\mathcal{L}(Y_t|Y_s=x)$ by the density
\begin{equation}\label{Tr-M}
f_{s,t}(y|x)=f\left(y;A-t, B-t, t-s+i\sqrt{x},t-s-i\sqrt{x}\right).
\end{equation}
It remains to verify that the above definitions are consistent.
\begin{proposition}\label{P-Y}
Formulas \eqref{f-M} and \eqref{Tr-M} determine a Markov process $(Y_t)_{t\in T}$.
Furthermore, $\E(Y_0)={(ABC+ABD+ACD+B C D)}/{(A+B+C+D)}$ by \eqref{m-abcd} and
\begin{equation}\label{m-M}
\E(Y_t)=\E(Y_0)+2\frac{A B-C D}{A+B+C+D} t-t^2.
\end{equation}
For $s\leq t$ in $T$,
\begin{equation}\label{cov}
\mbox{\rm Cov}(Y_s,Y_t)=M^2 (C+D+2 s) (A+B-2 t),
\end{equation}
where
\begin{equation}\label{SS}
M^2=\frac{(A+C) (B+C) (A+D) (B+D) }{(A+B+C+D)^2 (A+B+C+D+1)}>0.
\end{equation}

In view of \eqref{m-M}, the conditional moments simplify when we express them in terms of \begin{equation}\label{tilde y}
\widetilde Y_t=Y_{t/2}+t^2/4, \; -2\Re( C)<t<2\Re (A),
\end{equation}
with linear mean $\E(\widetilde Y_t)=\alpha+\beta t$ and the covariance $\mbox{\rm Cov}(\widetilde Y_s,\widetilde Y_t)=M^2 (C+D+ s) (A+B- t)$ for $s\leq t$. The one-sided conditional moments $s\leq t$ are:
\begin{equation}\label{m-M-c}
\E(\widetilde Y_t|\widetilde Y_s)=\frac{(A+B- t)}{A+B- s}\widetilde Y_s+\frac{A B (t-s)}{A+B- s}\,,
\end{equation}
\begin{equation}\label{m-V-c}
\Var(\widetilde Y_t|\widetilde Y_s)=
%\frac{2 (t-s) (A+B-2 t) \left((A-s)^2+Y_s\right) \left((B-s)^2+Y_s\right)}{(A+B-2 s)^2 (A+B-2 s+1)}.
\frac{ (A+B- t) (t-s) \left(A^2-s A+\widetilde Y_s\right) \left(B^2- s B+\widetilde Y_s\right)}{(A+B- s)^2 (A+B- s+1)}\,.
\end{equation}
For $s<t<u$ in $T$,
\begin{equation}\label{m-M-cc}
\E(\widetilde Y_t|\widetilde Y_s,\widetilde Y_u)= \frac{(u-t)  \widetilde Y_s+(t-s) \widetilde Y_u}{u-s}\,,
\end{equation}
\begin{equation}\label{m-V-cc}
\Var(\widetilde Y_t|\widetilde Y_s,\widetilde Y_u)=\frac{(u-t)(t-s)}{u-s+1}\left( \frac{(\widetilde Y_u-\widetilde Y_s)^2}{(u-s)^2}+ \frac{u\widetilde Y_s-s\widetilde Y_u}{u-s}\right).
\end{equation}
%$$
%-\frac{(s-t) (t-u) \left(\left((s-u)^2+Y_s\right)^2+Y_u^2+2 \left((s-u)^2-Y_s\right) Y_u\right)}{(2 s-2 u-1) (s-u)^2}
%$$
%$$
%=-\frac{(s-t) (t-u) (s-u)^2}{2 s-2 u-1}-\frac{2 (s-t) (t-u) z}{2 s-2 u-1}+x \left(\frac{2 (s-t) (t-u) z}{(2 s-2 u-1) (s-u)^2}
%-\frac{2 (s-t) (t-u)}{2 s-2
%   u-1}\right)$$
%   $$-\frac{(s-t) (t-u) x^2}{(2 s-2 u-1) (s-u)^2}
%   -\frac{(s-t) (t-u) z^2}{(2 s-2 u-1) (s-u)^2}
%$$
\end{proposition}
\begin{proof}
To verify Chapman-Kolmogorov equations
we first use  \eqref{CK1} with $m=t-s$, $a=A-t$, $b=B-t$, $c=C+t$, $d=D+t$.
This gives
\begin{equation}\label{Chapman1}
f_t(y)=\int_0^\infty f_{s,t}(y|x)f_s(x)\,dx.
\end{equation}

 Next we use   \eqref{CK1} with  $m=u-t$, $a=A-u$, $b=B-u$, $c=t-s+i\sqrt{x}$, $d=t-s-i\sqrt{x}$ to verify the Chapman-Kolmogorov equations  for the transition probabilities,
 \begin{equation}\label{Chapman2}
 f_{s,u}(z|x)=\int_0^\infty f_{s,t}(y|x)f_{t,u}(z|y)\,dy
 .
\end{equation}

Formula \eqref{f/q} can be now reinterpreted as the formula for the conditional distribution $\mathcal{L}(Y_t|Y_s=x,Y_u=z)$, given by
the density
\begin{multline}\label{f-cond}
g(y|x,z)=\frac{f_{t,u}(z|y)f_{s,t}(y|x)}{f_{s,u}(z|x)}
=f\left(y; u-t+i\sqrt{z},u-t-i\sqrt{z},t-s+i\sqrt{x},t-s-i\sqrt{x}\right).
\end{multline}
Since this is again expressed in terms of the same density \eqref{f}, the formulas for the conditional mean and conditional variance are recalculated from Proposition \ref{P-abcd}.
Finally, we use \eqref{m-M}, and
\begin{equation}\label{m-V}
\Var(Y_t)=M^2{(A+B-2 t) (C+D+2 t)},
\end{equation}
%\eqref{m-V}
which is calculated from \eqref{Var-abcd},
and \eqref{m-M-c}, to compute $\E(Y_sY_t)$ and we get \eqref{cov}.
\end{proof}
\begin{corollary}\label{C-beta4}
$(\widetilde Y_t)$ can be transformed into a quadratic harness with
covariance \eqref{EQ: cov} and the conditional variance \eqref{EQ: q-Var} with parameters
\begin{eqnarray}
\eta+\theta&=& \frac{(A+B+C+D)^2}{\sqrt{(A+C) (B+C) (A+D) (B+D) (A+B+C+D+1)}}\,,\\
\theta-\eta&=&\frac{(C-D)^2-(A-B)^2}{\sqrt{(A+C) (B+C) (A+D) (B+D) (A+B+C+D+1)}}\,,\\
\sigma=\tau&=&\frac{1}{A+B+C+D+1}\,,
\end{eqnarray}
and $\gamma=1-2\sqrt{\sigma\tau}=(A+B+C+D-1)/(A+B+C+D+1)$.
\end{corollary}
\begin{proof}
We apply Proposition \ref{P-Mobius} with parameters
$$\alpha = \frac{A B C+A D C+B D C+A B D}{A+B+C+D},\;  \beta =
   \frac{AB-CD}{A+B+C+D}, \;\epsilon = -1, \; \psi =   C+D, \;\delta = A+B.$$ The only non-zero parameters in \eqref{computed-cond-var} are   $\eta _0= \tau_0=1$.
\end{proof}
\begin{remark}\label{Rem:domain}
The quadratic harness is defined on the interval
$$
T'=\left(\frac{C+D-2 \Re(C)}{A+B+2 \Re(C)}, \frac{C+D+2 \Re(A)}{A+B-2 \Re(A)}\right)\,.
$$
In particular, $T'=(0,\infty)$ if $A=\bar B$ and $C=\bar D$.  It is plausible that by allowing transition probabilities and univariate laws with discrete components, this interval could be extended to $(0,\infty)$ in all cases when $\Re(A+B)>0$ and  $\Re(C+D)>0$.
\end{remark}

 \subsection{The admissible range of $\eta,\theta$}
 In this section we study which collections of parameters correspond to quadratic harnesses from the previous construction.
Given $\sigma=\tau,\gamma=1-2\sigma$  and $\eta$, $\theta$ such that $\eta+\theta\ne0$, without loss of generality we may assume that $\eta+\theta>0$. For if we can find a quadratic  harness $(X_t)$ for one such set of parameters, then $(-X_t)$ is a quadratic harness on the same time domain, with the same $\sigma,\tau,\gamma$, but with $-\eta,-\theta$ instead of $\eta,\theta$.

Once we restrict ourselves to the case $\eta+\theta>0$,  we want to know for which $\eta,\theta,\sigma=\tau,\gamma=1-2\sigma$  we can find  $A,B,C,D$  that satisfy  the equations
from Corollary \ref{C-beta4} and satisfy the constraints for the construction of the Markov process $(Y_t)$.
We will see that we can always find such  $A,B,C,D$ if either $\eta\theta\geq 0$ (which under the assumption $\eta+\theta>0$, is equivalent to $\eta\geq 0$, $\theta >0$ or vice versa) or
 when the sign of $\eta\theta$ is arbitrary but $\eta^2<4\sigma$ and $\theta^2<4\tau$.

To proceed, we rewrite the equations from Corollary \ref{C-beta4} in equivalent form:
\begin{eqnarray}
A+B+C+D&=&(1-\sigma)/\sigma \label{A+B+C+D}\,,\\
(A+C) (B+C) (A+D) (B+D)&=&\frac{(1-\sigma)^4}{(\eta +\theta )^2 \sigma ^3}\,,\\
(C-D)^2-(A-B)^2&=&\frac{(\theta-\eta ) (1-\sigma )^2}{(\eta +\theta ) \sigma ^2}\,.\label{ABCD3}
\end{eqnarray}

 \subsubsection{Hyperbolic case} We first show that quadratic harnesses exist  for any $\eta,\theta$ such that $\eta+\theta> 0$, $\eta^2<4\sigma$ and $\theta^2<4\sigma$.  This is because in this case the system of equations (\ref{A+B+C+D}-\ref{ABCD3}) is solved by two conjugate pairs $A=\bar B$ and $C=\bar D$ with $A,C$ given by
$$A=\Re(A)+\frac{i (1-\sigma ) \sqrt{4 \sigma -\eta ^2}}{2 (\eta +\theta ) \sigma },\;
 C=\frac{1-\sigma}{2\sigma}- \Re(A)+\frac{i (1-\sigma ) \sqrt{4 \sigma -\theta ^2}}{2(\eta +\theta ) \sigma
   }$$ with arbitrary $0<\Re(A)<\frac{1-\sigma}{2\sigma}$.

 The apparent non-uniqueness in this solution and in others is in fact illusory, as it corresponds to the translation of the time domain $T$. This translation does not affect neither the transition probabilities of the final quadratic harness, nor the final time domain, which by Remark \ref{Rem:domain} is $T'=(0,\infty)$.

\subsubsection{} Next we go over the remaining choices of pairs $(\eta,\theta)$, and confirm that in each case we can always find a quadratic harness when $\eta\theta\geq0$.

We first consider $\eta,\theta$ such that $\eta+\theta> 0$, $\eta^2<4\sigma$ and $\theta^2\ge4\sigma $.  Then quadratic harnesses exist if $4\sigma+\eta^2+2\eta\theta>0$.
Indeed,  in this case the system of equations (\ref{A+B+C+D}-\ref{ABCD3}) is solved  with one conjugate pair $A=\bar B$. The solutions are
$$
A=\Re(A)+\frac{i (1-\sigma) \sqrt{4 \sigma -\eta ^2}}{2 (\eta +\theta ) \sigma }\,,
$$
$$C=\frac{\left(\eta +\theta -\sqrt{\theta ^2-4 \sigma }\right) (1-\sigma )}{2 (\eta +\theta ) \sigma
   }-\Re(A)\,,
$$
   $$
   D=\frac{\left(\eta +\theta +\sqrt{\theta ^2-4 \sigma }\right) (1-\sigma )}{2 (\eta +\theta ) \sigma
   }-\Re(A)\,.
   $$
   The restriction $4\sigma+\eta^2+2\eta\theta>0$ guarantees that  $\theta^2-4\sigma<(\theta+\eta)^2$ so one
 can find $\Re(A)>0$ such that $C>0$; then $D>0$ follows automatically.

The restriction $4\sigma+\eta^2+2\eta\theta>0$  holds, in particular,    if $\eta\ge0$, as then $4\sigma+\eta^2+2\eta\theta>2\eta^2+2\eta\theta=2\eta(\eta+\theta)\geq0$.

We remark that the left endpoint of the time domain $T'$ here is $0$, see Remark \ref{Rem:domain}.  This is of interest, since for such domains the one-sided conditional moments \eqref{m-M-c} and \eqref{m-V-c} imply uniqueness of the quadratic harness.

 Finally, if $\eta+\theta>0$ are such that $\eta^2\geq 4\sigma$ and
 $\theta^2\geq 4\sigma$, then under the condition
 $\eta+\theta>\sqrt{\eta^2-4\sigma}+\sqrt{\theta^2-4\sigma}$  one
 can choose a small enough $A>0$ such that
  $$B=A+\frac{\sqrt{\eta ^2-4 \sigma } (1-\sigma)}{(\eta +\theta )
 \sigma }\,,$$
  $$C=\frac{\left(\eta +\theta -\sqrt{\eta ^2-4 \sigma }-\sqrt{\theta ^2-4 \sigma }\right) (1-\sigma)}{2
   (\eta +\theta ) \sigma }-A\,,$$
  $$D=\frac{\left(\eta +\theta -\sqrt{\eta ^2-4 \sigma }+\sqrt{\theta ^2-4 \sigma }\right) (1-\sigma)}{2
   (\eta +\theta ) \sigma }-A\,,$$
are all positive.
  In particular, if $\eta,\theta>0$ then $\eta>\sqrt{\eta^2-4\sigma}$ and $\theta>\sqrt{\theta^2-4\sigma}$ so the above solution will indeed give us a quadratic harness on a finite interval $T'$.

\section{Three parameter beta integral}\label{Sec:3par}
  For $a>0$ and $b,c$ real positive or a
 complex conjugate pair with positive real part, define the following
 density on $[0,\infty)$:  (See
\cite[(7.i)]{Askey:1989} or \cite[Section 1.3]{Koekoek-Swarttouw})
\begin{equation}
g(x;a,b,c)=\frac{|\Gamma(a+ i \sqrt{x})\Gamma(b+i\sqrt{x})\Gamma(c+i\sqrt{x})|^2}
{4\pi \Gamma(a+b)\Gamma(a+c)\Gamma(b+c)\sqrt{x}|\Gamma(2i\sqrt{x})|^2}\,.
\end{equation}
As previously, it is straightforward to use properties of the gamma
 function to get formulas for the mean $\mu$ and the variance $\sigma^2$,
\begin{equation}
\mu=ab+ac+bc,\; \sigma^2=(a+b) (a+c) (b+c).
\end{equation}

The  relevant version of \eqref{f/q} is
\begin{equation}\label{f/q3}
\frac{g(x;a+m,b,c) g\left(y;a,m+i \sqrt{x},m-i \sqrt{x}\right)}{g(y;a,b+m,c+m) }=f\left(x,m+i
   \sqrt{y},m-i \sqrt{y},b,c\right),
\end{equation}
where $x,y,m>0$.

Let $A\in\RR$ and let $B,C$ be  either real or a complex conjugate pair, and without loss of generality we assume that in the real case $B\geq C$.   Suppose in addition that  $A+\Re(C)>0$ so that  $T=(-\Re(C),A)$ is non-empty. Then from \eqref{f/q3} we get again a Markov process $(Y_t)_{t\in T}$ with
univariate distributions on the state space $(0,\infty)$ defined by the densities
$g(x;A-t, B+t,C+t)$, with transition probabilities defined for $s<t$ in $T$ and $x,y>0$ by the densities
$g(y; A-t, t-s -i\sqrt{x}, t-s+i\sqrt{x})$, and whose two-sided conditional laws are again
given by Wilson's density \eqref{f-cond}. In particular, after we make substitution \eqref{tilde y} formulas \eqref{m-M-cc} and \eqref{m-V-cc} for the two-sided conditional mean and variance hold.

As previously,  parameters $A,B,C$ affect only  the mean and the covariance of $(Y_t)$:
\begin{equation}
\E(Y_t)=-t^2+2 A t+AB+AC+B C,\; \Var(Y_t)=(A + B) (A + C) (B + C + 2 t)\,.
\end{equation}
Passing to the centered process \eqref{tilde y}, the one-sided conditional moments are:
$$
\E(\widetilde Y_t|\widetilde Y_s)=  A (t-s)+\widetilde Y_s,\; \V(\widetilde Y_t|\widetilde Y_s)= (t-s) \left(A^2- s A+\widetilde Y_s\right)\,.$$
In particular, the above formula for $\E(\widetilde Y_t|\widetilde Y_s)$ gives
$${\rm Cov}(\widetilde Y_s,\widetilde Y_t)=(A + B) (A + C) (B + C +  \min\{t,s\})\,.
$$
Then the transformation from  Proposition \ref{P-Mobius} takes a particularly simple form.
Markov process $$X_t=\frac{\widetilde Y_{t-B-C}-\E( \widetilde Y_{t-B-C})}{\sqrt{(A + B) (A + C)}}$$ is
a quadratic harness %with covariance \eqref{EQ: LR} and
with parameters
\begin{eqnarray}
\eta&=&\frac{1}{\sqrt{(A+B) (A+C)}} \,,\\
\theta&=&\frac{2 A+B+C}{\sqrt{(A+B) (A+C)}}\,,\\
\sigma&=&0 \,,\\
\tau&=&1\,.
\end{eqnarray}
This gives us a family of quadratic harnesses arbitrary positive values for parameters $\eta,\theta$, with $\tau=1$, $\sigma=0$.
Other values of parameters are now produced by routine transformations that were mentioned in the introduction.
To swap the roles of $\sigma,\tau$ one uses time inversion $(tX_{1/t})$. Taking $(-X_t)$ we get arbitrary negative values of $\eta,\theta$, covering all possible  non-zero values of the same sign ($\eta\theta>0$). Finally, transformation  $(X_{\alpha t}/\sqrt{\alpha})$ produces arbitrary positive values for parameter $\tau$.

\begin{remark} The quadratic harness is defined on
$$
T'=(\Re(C-B),\infty).
$$
In particular, $T'=(0,\infty)$ if $B=\bar C$.  It would be interesting to see if the construction could be modified to yield $T'=(0,\infty)$ also for real $B\ne C$.

\end{remark}

\begin{remark}
Formula \eqref{f/q3} indicates that bridges of the three-parameter quadratic harnesses with $\sigma=0$ are the (transformations of) four-parameter quadratic harnesses from Corollary \ref{C-beta4}. It would be interesting to see if this holds true also in the cases without densities.
\end{remark}

\section{Two-parameter beta integral}\label{Sec:2par}
According to \cite[(5.i)]{Askey:1989}, see also \cite[Section 1.4]{Koekoek-Swarttouw}, the following is a probability density on $\RR$ when $c=\bar a, d=\bar b$ have positive real part.
\begin{equation}\label{varphi-denisty}
\varphi(x;a,b,c,d)=\frac{\Gamma (a+b+c+d) \Gamma
   \left(a+i {x}\right) \Gamma \left(b+i {x}\right) \Gamma \left(c-i {x}\right) \Gamma \left(d-i {x}\right)}{2\pi\Gamma (a+c) \Gamma (b+c) \Gamma
   (a+d) \Gamma (b+d)}.
\end{equation}

 The analog of Proposition \ref{P-abcd} is
 \begin{proposition}\label{P-cont_Hahn}
 If a random variable $X\in\RR$ has density $\varphi(x;a,b,c,d)$, then
  \begin{equation}\label{2-par-mean}
  \E(X)=-\frac{\Re (a) \Im (b)+\Re(b)\Im(a)}{\Re(a+b)}\,,
\end{equation}
\begin{equation}\label{2-par-var}
\Var(X)=\frac{\Re(a) \Re(b)\left((\Re(a+b))^2+(\Im(a-b))^2\right)}{(\Re(a+b))^2 (2 \Re(a+b)+1)}\,.
\end{equation}

\end{proposition}
\begin{proof}
Denote by $K(a,b,c,d)=\frac{\Gamma (a+b+c+d) }{2\pi\Gamma (a+c) \Gamma (b+c) \Gamma
   (a+d) \Gamma (b+d)}$ the normalizing constant in \eqref{varphi-denisty}. Then
\begin{multline}
\int_{-\infty}^\infty x \varphi(x;a,b,c,d)\,dx\\
=\frac{1}{i(c+b-a-d)}\left(\int_{-\infty}^\infty \left((a+ix)(c-ix)-(b+ix)(d-i x) +bd-ac\right)\varphi(x;a,b,c,d)\,dx\right)\\
=\frac{K(a,b,c,d)}{i(c+b-a-d)}\left(\frac{1}{K(a+1,b,c+1,d)} -\frac{1}{K(a,b+1,c,d+1)}+bd-ac\right)=
\frac{i (a b-c d)}{a+b+c+d}.\end{multline}
Substituting $a=\Re(a)+i\Im(a)$, $b=\Re(b)+i\Im(b)$, $c=\Re(a)-i\Im(a)$, $d=\Re(b)-i\Im(b)$
we get \eqref{2-par-mean}.

The variance comes from a similar calculation:
\begin{multline*}
\E(X^2)-(\E(X))^2=K(a+1,b,c+1,d)-(c-a)\E(X)-ac-(\E(X))^2\\=\frac{(a+c) (b+c) (a+d) (b+d)}{(a+b+c+d)^2 (a+b+c+d+1)}.
\end{multline*}

\end{proof}

 The analog of Proposition \ref{P-CK} is based on the identity
 \begin{equation}\label{f/q2}
\frac{\varphi(y; a, m - i{x}, \bar{a}, m + i {x})
\varphi(x; a + m, b, \bar{a} + m, \bar{b})}{
 \varphi(y; a, b + m, \bar{a}, \bar{b} + m)}=
\varphi(x; b, m - i {y}, \bar{b}, m + i {y}).
\end{equation}
%\comment{There might be more such representations?}
Thus, given complex parameters $A,B$, such that $\Re(A+B)>0$, let $T= (-\Re(B),\Re(A))$. For $s<t$ in $T$, the univariate densities on $\RR$
\begin{equation}\label{two-par-univ}
f_t(x)=\varphi(x, A-t, B+t,\bar{A}-t,\bar{B}+t)\,,
\end{equation}
and the transition probabilities
\begin{equation}\label{two-param-transitions}
f_{s,t}(y|x)=\varphi(x, A-t, t-s-i{x}, \bar{A}-t, t-s+i{x})\,,
\end{equation}
satisfy the Chapman-Kolmogorov equations \eqref{Chapman1} and \eqref{Chapman2}.
 Let $(Y_t)_{t\in T}$ denote the corresponding Markov process.
 %\comment{Write $A=\alpha+i\gamma$, $B=\beta+i\delta$?}
 Then from \eqref{2-par-mean} and \eqref{2-par-var} we get
 \begin{equation}\label{2-param-Y}
 \E(Y_t)=\frac{\Im(B-A)}{\Re(A+B)}t-\frac{\Re(A) \Im(B)+\Im(A)\Re(B)}{\Re(A+B)},\;
  \V(Y_t)=M^2
 ({\Re(A)}-t)
   ({\Re(B)}+t),
\end{equation}
where
\begin{equation}\label{2par-M}
M^2=\frac{({\Im(A-B)})^2+({\Re(A+B)})^2 }{({\Re(A+B)})^2 (2 {\Re(A+B)}+1)}\,.
\end{equation}
Since  \eqref{2-par-mean} also gives
$$
\E(Y_t|Y_s)=\frac{{\Re(A)}-t }{{\Re(A)}-s}Y_s-\frac{{\Im(A)} (t-s)}{{\Re(A)}-s}
$$
for $s<t$, from \eqref{2-param-Y} we further calculate
\begin{equation}\label{2-param-Y-Cov}
 {\rm Cov}(Y_s,Y_t)=M^2({\Re(A)}-t)
   ({\Re(B)}+s).
\end{equation}

Next we compute conditional moments. For $s<t<u$ in, the two-sided conditional density  of  $\mathcal{L}(Y_t|Y_s=x,Y_u=z)$ is given by
 $$
  g(y|x,z)=\varphi(y; t-s-i x, u-t-iz, t-s+i x,u-t+iz)\,.
  $$

  So from \eqref{2-par-mean},
   $$\E(Y_t|Y_s,Y_u)=\frac{(u-t) Y_s+(t-s) Y_u}{u-s}\,.$$
  and from \eqref{2-par-var} we get
\begin{equation}\label{2par-C-Var}
\Var(Y_t|Y_s,Y_u)=\frac{(t-s) (u-t)}{(2 u-2 s+1)} \left(1+\frac{\left(Y_u-Y_s\right)^2}{ (u-s)^2}\right)\,.
\end{equation}

From Proposition \ref{P-Mobius} applied with
$ \chi _0= 1$,
$\alpha = -\frac{{\Im(B)} {\Re(A)}+{\Im(A)}
   {\Re(B)}}{{\Re(A+B)}}$, $\beta =
   \frac{{\Im(B-A)}}{{\Re(A+B)}}$, $\eta _0= 0$, $\theta _0= 0$,
   $\epsilon = -1$, $\psi = {\Re(B)}$, $\delta = {\Re(A)}$,
   $M=
   \frac{\sqrt{({\Im(A-B)})^2+({\Re(A+B)})^2}}{\Re(A+B) \sqrt{2
   {\Re(A+B)}+1}}$,
%   ,$\phi _1=
%   \frac{({ImA}-{ImB})^2}{({ReA}+{ReB})^2}+1,\eta '= 0,\theta '= \frac{2
%   ({ImB}-{ImA})}{({ReA}+{ReB})
%   \sqrt{\frac{({ImA}-{ImB})^2}{({ReA}+{ReB})^2}+1} \sqrt{\kappa }}$
we get
$$
\sigma=\tau=\frac{1}{2 {\Re(A+B)}+1}\,,
$$
$$
\eta=-\theta=\frac{2 ({\Im(A-B)})}{\sqrt{(2 {\Re(A+B)}+1)
   \left((\Im(A-B))^2+(\Re(A+B))^2\right)}}\,.
$$
 From the first equation, we see that $\Re(A+B)=\frac{1-\sigma}{2\sigma}$. The second equation determines $\Im(A-B)$ as  a real number iff $\theta^2<4\tau$. This proves the following.
 \begin{proposition}
   For every $\sigma\in(0,1)$ and $\eta\in(-2\sqrt{\sigma},2\sqrt{\sigma})$, there is a quadratic harness on $(0,\infty)$ with parameters $\eta, \theta=-\eta, \sigma, \tau=\sigma,\gamma=1-2\sigma$.
 \end{proposition}

\subsection{Bridges of  the hyperbolic secant  process}
%The two-parameter processes turn out as  bridges in  the hyperbolic secant  processes. This is worth working out in more detail, as from \cite[Proposition 4.1]{Bryc-Wesolowski-09} we know that all such bridges have
%$\eta\sqrt{\tau}+\theta\sqrt{\sigma}=0$, a "boundary case" that is not covered by the construction in Section \ref{Sec:4par}.
%\section{Bridges of Meixner processes}
Bridges of all Meixner processes are described in
 \cite[Proposition 4.2 and Remark 4.1]{Bryc-Wesolowski-09}. According to these results, bridges of Meixner processes are quadratic harnesses with $\eta\sqrt{\tau}+\theta\sqrt{\sigma}=0$. When $\sigma\tau>0$, then depending on the sign of $\theta^2-4\tau$, such processes arise as bridges of negative binomial, gamma, or the hyperbolic secant  process.    In   \cite{Bryc-Wesolowski-09} the bridges of hyperbolic secant  process were not described explicitly, so we identify their transition probabilities here.

The following integral is  due to Meixner \cite[page 13]{Meixner:1934}, and is listed as \cite[(4.i)]{Askey:1989}:
\begin{equation}
\int_{-\infty}^\infty |\Gamma(a+ ix)|^2 e^{\beta x} dx=\frac{2\pi \Gamma(2 a)}{(2\cos\frac \beta 2)^{2 a}}.
\end{equation}
The integral is well defined for real $a>0$ and $-\pi<\beta<\pi$.
Denote by $f(x;a,\beta)$ the corresponding density, i.e.
\begin{equation}\label{HM}
f(x;a,\beta) = \frac{(2\cos\frac \beta 2)^{2a}}{2\pi \Gamma(2 a)}  |\Gamma(a+ ix)|^2 e^{\beta x}\,,
\end{equation}
and by $X$ the corresponding random variable.

%\comment{Another form: $f(x;t,\beta) = \frac{(1+\cos \beta )^{t}}{2\pi \Gamma(2 t)}  |\Gamma(t+ ix)|^2 e^{\beta x}$.}

Differentiating \eqref{HM} with respect to $\beta$ and integrating the answer we get
$\E(X)=a \tan \left(\frac{\beta}{2}\right)$ and $\Var(X)=\frac{1}{2} a \sec ^2\left(\frac{\beta}{2}\right)$.
It is known that the corresponding Markov process  has independent increments:
the univariate law of $Y_t$  has density
$f_t(x)=f(x;A-t, \beta)$ and the transition densities are
$f_{s,t}(y|x)=f(y-x;t-s,\beta)$.
One can verify also Chapman-Kolmogorov equations directly from the analog of Proposition \ref{P-CK} which is based on the identity
\begin{equation}\label{BBBridge}
\frac{f(y-x; m,\beta)f(x;a,\beta)}{f(y; a+m,\beta)}=\varphi(x; a, m - i y, a, m + i y).
\end{equation}
The right hand side of  \eqref{BBBridge} integrates to $1$ because of
 \eqref{varphi-denisty}.
(This gives an "elementary" proof of the well known fact established
 by Laha and Lukacs \cite[Lemma 2]{Laha-Lukacs60} that  the hyperbolic secant  laws form a
 convolution semi-group.)

The following proposition describes in more detail bridges mentioned in \cite[Remark 4.1]{Bryc-Wesolowski-09}.
\begin{proposition}\label{P-M_bridges}
All bridges of a hyperbolic secant  process are described by formulas
\eqref{two-par-univ} and \eqref{two-param-transitions}. Conversely,
all  quadratic harnesses with $0<\sigma\tau<1$, $\gamma=1-2\sqrt{\sigma\tau}$ and
$\eta,\theta\in\RR$ such that $\eta\sqrt{\tau}+\theta\sqrt{\sigma}=0$ and $\theta^2<4\tau$ can be realized as such bridges.
\end{proposition}
\begin{proof}
From \eqref{BBBridge} we see that for a hyperbolic secant  process
 $(Y_t)$,  the two-sided conditional law
 of  $\mathcal{L}(Y_t|Y_s=x,Y_u=z)$ is given by
 \begin{equation}\label{Meixner bridge}
   g(y|x,z)=\varphi(y-x; t-s, u-t-i(z-x), t-s,u-t+i(z-x)).
\end{equation}
Inspecting formula \eqref{varphi-denisty}, we see that $$\varphi(y-x; t-s, u-t-i(z-x), t-s,u-t+i(z-x))=\varphi(y; t-s-ix,
 u-t-iz, t-s+ix,u-t+i z)$$
$$
=\varphi(y; u-t-iz, t-s-ix, u-t+i z,t-s+ix)\,.
$$
So identifying this with the univariate law of the bridge at time
 $S<t<U$, conditioned at $S<U$, we can read out that the bridge  corresponds to the Markov process with transition probabilities \eqref{two-param-transitions}, where
$A=U- i Y_U$, $B=-S-iY_S$.
%In particular, from \eqref{2-par-mean}   we get
%$$\E(Y_t|Y_s,Y_u)%=\frac{1}{2} \left((-2 t+2 u+1) Y_s+(2 t-2 u+1) Y_u\right)
%=Y_s+\frac{t-s}{u-s}(Y_u-Y_s)=\frac{(u-t) Y_s+(t-s) Y_u}{u-s}$$
%and from \eqref{2-par-var} we get
%$$\Var(Y_t|Y_s,Y_u)=\frac{(t-s) (u-t)}{(2 u-2 s+1)} \left(1+\frac{\left(Y_u-Y_s\right)^2}{ (u-s)^2}\right)
%$$

\end{proof}
\section{Standard beta integral}\label{Sect:SBI}

In this section we use the well known beta  density
\begin{equation}\label{beta1}
f(x;a,b)=\frac{\Gamma(a+b)}{\Gamma(a)\Gamma(b)}x^{a-1}(1-x)^{b-1}
\end{equation}
 to re-derive the quadratic harness properties of the one-parameter family of Dirichlet processes from \cite[Eample 4.1]{Bryc-Wesolowski-09}.  (Here,  the density is on $0<x<1$, and the parameters satisfy  $a,b>0$.)
It is well know that the corresponding random variable $X$ has moments
\begin{equation}\label{beta}
\E(X)=\frac{a}{a+b},\quad
%\E(X^2)=\frac{a(a+1)}{(a+b)(a+b+1)},\\
\mbox{and}\quad
\Var(X)=\frac{ab}{(a+b)^2(a+b+1)}\;.
\end{equation}
The  analog of \eqref{f/q} is the algebraic identity
\begin{equation}\label{f/q5}
\frac{\frac{1}{1-x}f\left(\frac{y-x}{1-x};m,b\right)f(x;a,b+m)}{f(y; a+m,b) }=\frac1yf(x/y;a,m).
\end{equation}
In particular, we have a "convolution formula",
\begin{equation}\label{CK5}
\int_0^y \frac{1}{1-x}f\left(\frac{y-x}{1-x};m,b\right)f(x;a,b+m) dx=f(y; a+m,b).
\end{equation}
Given $A>0$, we now use \eqref{CK5} to  define the Markov process $(Y_t)_{0<t<A}$ by specifying its univariate laws as
$$
f_t(x)=f(x;t,A-t), \; x\in[0,1],
$$
and for $s<t$, $y\geq x$ its transition probabilities as
$$
f_{s,t}(y|x)=\frac{1}{1-x}f\left( \frac{y-x}{1-x};t-s, A-t\right).
$$
A calculation based on \eqref{CK5} shows that these expressions indeed satisfy
the Chapman-Kolmogorov equations, so Markov process $(Y_t)_{t\in(0,A)}$ is well defined.
(The same conclusion can be reached via probabilistic arguments, as these processes arise as bridges of the gamma process.)

From \eqref{beta}, $\E(Y_t)=t/A$ and $\Var(Y_t)=\frac{t(A-t)}{A^2(A+1)}$, and with some more work one can read out that $ {\rm Cov}(Y_s,Y_t)=\frac{s(A-t)}{A^2(A+1)}$ for $s\leq t$.

Since \eqref{f/q5} shows that two-sided conditional laws are also beta, from \eqref{beta} we can read out the conditional moments
$$
\E(Y_t-Y_s|Y_s,Y_u)= \frac{t-s}{u-s}(Y_u-Y_s),
$$
$$
\Var(Y_t|Y_s,Y_u)=\frac{(t-s)(u-t)}{(u-s)^2((u-s)+1)}(Y_u-Y_s)^2.
$$
Applying Proposition \ref{P-Mobius} with $M= \frac{1}{A \sqrt{A+1}}$, $\beta =   {1}/{A}$, $
 = A$, $\epsilon \to -1$ (the remaining parameters are $0$), we see that $(Y_t)$ can be transformed into a quadratic harness on $T'=(0,\infty)$ with parameters
$$
\eta=-\theta=-2/\sqrt{1+A},\; \sigma=\tau=1/(1+A), \; \gamma=1-2/(1+A)^2.
$$
(This is consistent with \cite[Eample 4.1]{Bryc-Wesolowski-09}.)

\subsection*{Acknowledgement}
We would like to thank Arthur Krener and Ofer Zeitouni for information on reciprocal processes, and to J. Weso\l owski for  several related discussions.
This research was partially supported by NSF
grant \#DMS-0904720, and by the Taft Research Center.

%
%%BibTeX
%%\bibliographystyle{apalike}
%\bibliographystyle{plain}
\bibliographystyle{alpha}
\bibliography{../Vita,W-09,reciprocal}

\end{document}